\documentclass[11pt]{article} 
\usepackage{latexsym}
\usepackage{graphics,epsf,psfrag}
\usepackage{graphicx}
\usepackage{amssymb}
\usepackage{amsmath}

\newtheorem{definition}{Definition}
\newtheorem{lemma}{Lemma}
\newtheorem{theorem}{Theorem}
\newtheorem{claim}{Claim}

\newtheorem{proposition}{Proposition}
\newtheorem{corollary}{Corollary}

\newenvironment{proof}{{\bf Proof.}}{\hfill $\Box$ \bigskip}

\title{Global Behavior of Solutions to Two Classes of Second Order Rational Difference Equations}
\author{{Sukanya Basu} \quad {Orlando Merino}\\
University of Rhode Island, Kingston, RI 02881\\
{\it sukanya@math.uri.edu \quad merino@math.uri.edu}}
\date{\ }

\begin{document}
\maketitle
\begin{abstract}
For nonnegative real numbers 
$\alpha$, $\beta$, $\gamma$, $A$, $B$ and $C$
such that $B+C>0$ and $\alpha+\beta+\gamma >0$, 
the difference equation
 \begin{equation*}
x_{n+1}=\displaystyle\frac{\alpha +\beta\,x_{n}+\gamma\,x_{n-1}}{A+B\,x_{n}+C\,x_{n-1}}\,, \quad  n=0,1,2,\dots
\end{equation*}
has a unique positive equilibrium.
A proof is given here for the following statements:
\medskip

\noindent
Theorem 1. {\it For every choice of positive parameters
$\alpha$, $\beta$, $\gamma$, $A$, $B$ and $C$,
all solutions to the difference equation
 \begin{equation*}
x_{n+1}=\displaystyle\frac{\alpha +\beta\,x_{n}+\gamma\,x_{n-1}}{A+B\,x_{n}+C\,x_{n-1}}\,, \quad  n=0,1,2,\dots, \quad   x_{-1},x_{0}\in [0,\infty) 
\end{equation*}
converge 
to the positive equilibrium or to a prime period-two solution.}
\medskip

\noindent
Theorem 2. {\it For every choice of positive parameters
$\alpha$, $\beta$, $\gamma$, $A$, $B$ and $C$,
all solutions to the difference equation
\begin{equation*}
 x_{n+1}= \displaystyle\frac{\alpha +\beta\,x_{n}+\gamma\,x_{n-1}}{B\,x_{n}+C\,x_{n-1}}  \,, \quad n=0,1,2,\dots, \quad   x_{-1},x_{0}\in (0,\infty) 
\end{equation*}
converge 
to the positive equilibrium or to a prime period-two solution.}
\end{abstract}
 \tableofcontents
\vfill
\newpage 

\section{Introduction and Main Results}
\label{sec: Introduction}
In their book \cite{KuL01}, M. Kulenovi\'c and G. Ladas initiated a systematic study of the difference equation
 \begin{equation}
 \label{eq: 3-3 orig.}
x_{n+1}=\displaystyle\frac{\alpha +\beta\,x_{n}+\gamma\,x_{n-1}}{A+B\,x_{n}+C\,x_{n-1}}\,, \quad  n=0,1,2,\dots
\end{equation}
for  nonnegative real numbers 
$\alpha$, $\beta$, $\gamma$, $A$, $B$ and $C$
such that $B+C>0$ and $\alpha+\beta+\gamma >0$,
and for nonnegative or positive initial conditions $x_{-1}$, $x_0$.
Under these conditions, (\ref{eq: 3-3 orig.}) has a unique positive equilibrium.
One of their main ideas in this undertaking was to
make the task more manageable by considering 
separate cases when one or more of the parameters in (\ref{eq: 3-3 orig.}) 
is zero.  The need for this strategy is made apparent by cases such as 
the well known {\it Lyness Equation}  \cite{Ladas95}, \cite{Zeeman}, \cite{Kulenovic}.
\begin{equation}
x_{n+1} = \frac{\alpha + x_n}{x_{n-1}}
\end{equation}  
whose dynamics differ significantly from other equations in this class.
There are a total of 42 cases that arise  from (\ref{eq: 3-3 orig.}) in the manner just discussed,
under the hypotheses $B+C>0$ and $\alpha+\beta+\gamma >0$.
The recent publications \cite{Ladas-07-1}, \cite{Ladas-07-2} give a detailed account
of the progress up to 2007 in the study of dynamics of the class of equations (\ref{eq: 3-3 orig.}).
After a sustained effort by many researchers
(for extensive references, see \cite{Ladas-07-1}, \cite{Ladas-07-2}), there are some cases that have resisted a complete analysis.
We list them below in normalized form, as presented in \cite{Ladas-07-1}, \cite{Ladas-07-2}.
\begin{eqnarray}
x_{n+1} & = & \frac{\alpha+x_n}{A+ x_{n-1}}  \label{eq: ladas y2k} \\
x_{n+1} & = & \frac{\alpha+x_n}{x_n+C\, x_{n-1}} \\
x_{n+1} & = & \frac{\alpha +\beta\,x_{n}+\gamma\,x_{n-1}}{x_{n-1}} \label{eq: kul 1} \\
x_{n+1} & = & \frac{\alpha + x_{n}}{A+B\,x_{n}+x_{n-1}}  \label{eq: kul} \\
x_{n+1} & = & \frac{\beta\,x_{n}+x_{n-1}}{A+B\,x_{n}+ x_{n-1}} \label{eq: ladas 2-3} \\
x_{n+1} & = & \frac{\alpha +\beta\,x_{n}+\gamma\,x_{n-1}}{A+ x_{n-1}} \\
x_{n+1} & = & \frac{\alpha + x_{n}+\gamma\,x_{n-1}}{ B\,x_{n}+ x_{n-1}} \label{eq: ladas 3-2} \\
x_{n+1} & = & \frac{\alpha +\beta\,x_{n}+ x_{n-1}}{A+B\,x_{n}+ x_{n-1}}   \label{eq: ladas 3-3}
\end{eqnarray}
The dynamics of Equation (\ref{eq: ladas 2-3}) has been settled recently in \cite{BK}, \cite{Stevic}.
Global attractivity of the positive equilibrium  of Equation (\ref{eq: ladas y2k}) has been 
proved recently in \cite{Y2K}. Since Eq.(\ref{eq: kul}) 
can be reduced to Eq.(\ref{eq: ladas y2k}) through
a change of variables \cite{KuLMR03}, 
global behavior of solutions to (\ref{eq: kul}) is also settled.
Equation (\ref{eq: kul 1}) is another equation that can be reduced to (\ref{eq: ladas y2k}),
through the change of variables $x_n = y_n+\gamma$ \cite{Mustafa told me this one}.

Ladas and co-workers \cite{KuL01}, \cite{Ladas-07-1}, \cite{Ladas-07-2}, have posed a series of conjectures on these equations.
One of them is the following.
\medskip

\noindent {\bf Conjecture [Ladas et al.]}\ 
{\it For equations (\ref{eq: ladas 3-2}) and (\ref{eq: ladas 3-3}), every solution converges to 
the positive equilibrium or to a prime period-two solution.
}
\medskip

In this article, we prove this conjecture.  Our main results are the following.
\medskip

\begin{theorem}
\label{th: 3-3}
For every choice of positive parameters
$\alpha$, $\beta$, $\gamma$, $A$, $B$ and $C$,
all solutions to the difference equation
 \begin{equation}
\tag{3-3}
x_{n+1}=\displaystyle\frac{\alpha +\beta\,x_{n}+\gamma\,x_{n-1}}
{A+B\,x_{n}+C\,x_{n-1}}\,, \quad  n=0,1,2,\dots, \quad   x_{-1},x_{0}\in [0,\infty) 
\end{equation}
converge 
to the positive equilibrium or to a prime period-two solution.
\end{theorem}
\medskip

\begin{theorem}
\label{th: 3-2 orig.}
For every choice of positive parameters
$\alpha$, $\beta$, $\gamma$, $A$, $B$ and $C$,
all solutions to the difference equation
\begin{equation}
\label{eq: 3-2 orig.}
 x_{n+1}= \displaystyle\frac{\alpha +\beta\,x_{n}+\gamma\,x_{n-1}}
 {B\,x_{n}+C\,x_{n-1}}  \,, \quad n=0,1,2,\dots, \quad   x_{-1},x_{0}\in (0,\infty) 
\end{equation}
converge 
to the positive equilibrium or to a prime period-two solution.
\end{theorem}

A reduction of the number of parameters of Eq.(\ref{eq: 3-2 orig.}) is obtained with the change of variables
$x_n = \frac{\gamma}{C}\, y_n$, which yields the equation
\begin{equation}
\tag{3-2}
y_{n+1}= \displaystyle\frac{r +p\,y_{n}+ y_{n-1}}
 {q\,y_{n}+ y_{n-1}}  \,, \quad n=0,1,2,\dots, \quad   y_{-1},y_{0}\in (0,\infty) 
\end{equation}
where $r=\frac{\alpha\,C}{\gamma^2}$, $p=\frac{\beta}{\gamma}$, and 
$q = \frac{B}{C}$.

The number of parameters of Eq.(3-3) can also be reduced, 
which we proceed to do next.
Consider the following  affine change of variables 
which is helpful to reduce number of parameters and simplify calculations:
\begin{equation}
\label{change of coordinates}
x_n  = \left( \frac{\gamma}{C} + \frac{A}{B+C}\right) \,y_n - \frac{A}{B+C}\, ,
\end{equation}
With (\ref{change of coordinates}), Eqn.(3-3) may now be rewritten as 
  \begin{equation*}
\tag{3-2-L}
  y_{n+1}=\displaystyle \frac{r+p\, y_{n}+y_{n-1}}{q\, y_{n}+y_{n-1}}\, , \quad 
n=0,1,2,\dots, \,\, y_{-1},y_{0} \in [L,\infty) 
\end{equation*}
where
\begin{equation}
\label{eq: pqr}
\begin{array}{rcl}
p &=&  \frac{A\,B + (B+C)\,\beta}{A\,C +(B+C)\,\gamma} \\ \\
q &=&  \frac{B}{C}, \\ \\
r &=&  \frac{C(B+C)(B\,\alpha+C\,\alpha-A\,\beta-A\,\gamma)}{(\gamma\,( B+C)+A\,C)^{2}}, \\ \\
L&=&   \frac{A\,C}{A\,C + (B + C)\,\gamma}
\end{array}
\end{equation}


Theorems \ref{th: 3-3} and \ref{th: 3-2 orig.} can be 
 reformulated in terms of the parameters $p$, $q$ and $r$ as follows.
\begin{theorem}
\label{th: 3-2-L}
Let $\alpha$, $\beta$, $\gamma$, $A$, $B$ and $C$ be positive numbers, and let
$p$, $q$, $r$ and $L$ be given by relations (\ref{eq: pqr}).
Then every solution to Eqn.(3-2-L)
converges to the unique equilibrium  
or to a prime period-two solution. 
\end{theorem}
\begin{theorem}
\label{th: 3-2}
Let  
$p$, $q$, $r$ be positive numbers.
Then every solution to Eqn.(3-2)
converges to the unique equilibrium  
or to a prime period-two solution. 
\end{theorem}
In this paper we prove Theorems \ref{th: 3-2-L} and \ref{th: 3-2};
Theorems \ref{th: 3-3} and  \ref{th: 3-2 orig.} follow 
as an immediate corollary.

The two main differences between Eq.(3-2-L) and Eq.(3-2) are
the set of initial conditions, and the possibility of having 
a negative value of $r$ in Eq.(3-2-L), while only positive values 
of $r$ are allowed in Eq.(3-2).
Nevertheless,  for both Eq.(3-2-L) and Eq.(3-2) the unique equilibrium has the  formula:
\begin{equation*}
\overline{y} = \frac{p+1 + \sqrt{(p+1)^2+ 4\,r\,(q+1)}}{2\,(q+1)}
\end{equation*}
Although it is not possible to prove Theorem \ref{th: 3-3} as a simple corollary to Theorem \ref{th: 3-2 orig.}, 
the changes of variables leading to Theorems \ref{th: 3-2-L} and \ref{th: 3-2} will result in proofs
to the former theorems that are greatly simplified.
\medskip

Our main results Theorem \ref{th: 3-3} and  Theorem \ref{th: 3-2 orig.} 
imply that when prime period-two solutions to Eq.(3-3) or Eq.(3-2)
do not exist, then the unique equilibrium is a global attractor.
We have not treated here certain questions about the global dynamics of Eq.(3-3) and Eq.(3-2),
such as the character of the prime period-two solutions to either equation, 
or even for more general rational second order equations, when 
such solutions exist.  This matter will be treated in an upcoming article of the authors \cite{BasuMerinoP2}.
\medskip

This work is organized as follows.
The main results are stated in Section \ref{sec: Introduction}.
Results from the literature which are used here are given in
Section \ref{sec: Results from the literature} for convenience.
In Section \ref{section: Existence}, it is shown that 
either every solution to Eq.(3-2-L) converges to the equilibrium,
or
there exists an invariant and attracting interval $I$ with the property
that the function $f(x,y)$ associated with the difference equation
is coordinate-wise strictly-monotonic on $I\times I$. 
In Section \ref{sec: The Equation (3-2)}, a global convergence result is obtained for 
Eq.(3-2) over a specific range of parameters and for 
initial conditions in an invariant compact interval.
Theorem \ref{th: 3-2-L} is proved in Section \ref{sec: Proof 3-2-L},
and the proof of Theorem \ref{th: 3-2} is given in Section \ref{sec: Proof 3-2}.
Section \ref{appendix: computer 1} includes computer algebra system code 
for performing certain calculations that involve
polynomials with a large number of terms (over 365,000 in one case).
These computer calculations are used to support certain statements in Section \ref{sec: The Equation (3-2)}.
Finally, we refer the reader to \cite{KuL01} for terminology and definitions
that concern difference equations.

\section{Results from the literature}
\label{sec: Results from the literature}
The results in this subsection are from the literature,
and they are given here for easy reference.
The first result is a reformulation of Theorems (1.4.5) --- (1.4.8) in \cite{KuL01}.
\begin{theorem}[ \cite{KLS}, \cite{KuL01}]
\label{th: mM first}
Suppose a continuous function $f : [a,b]^2 \rightarrow [a,b]$  satisfies one of i.--iv.:
 \begin{itemize}
 \item[i.] $f(x,y)$ is nondecreasing in $x$, $y$, and 
 $$\forall (m,M) \in [a,b]^2, \quad (\, f(m,m)=m\ \&\ f(M,M)=M\, ) \implies m=M
 $$ 
 \item[ii.] $f(x,y)$ is nonincreasing in $x$, $y$, and 
 $$\forall (m,M) \in [a,b]^2, \quad (\, f(m,m)=M\ \&\ f(M,M)=m\, ) \implies m=M
 $$ 
  \item[iii.] $f(x,y)$ is nonincreasing in $x$ and nondecreasing in $y$, and 
 $$\forall (m,M) \in [a,b]^2, \quad (\, f(m,M)=M\ \&\ f(M,m)=m\, ) \implies m=M
 $$ 
 \item[iv.] $f(x,y)$ is nondecreasing in $x$ and nonincreasing in $y$, and 
 $$\forall (m,M) \in [a,b]^2, \quad (\, f(M,m)=M\ \&\ f(m,M)=m\, ) \implies m=M
 $$ 
\end{itemize}
Then \quad $y_{n+1} = f(y_n,y_{n-1})$ 
 has a unique equilibrium in $ [a,b]$,   
 and every solution with initial values in $[a,b]$ converges to the equilibrium.
 \end{theorem}
 
 The following result is Theorem A.0.8 in \cite{KuL01}.
  \begin{theorem}
 \label{th: Mm 3-3}
 Suppose a continuous function $f : [a,b]^3 \rightarrow [a,b]$    
 is nonincreasing in all variables, and  
 $$
 \forall (m,M) \in [a,b]^3, \quad (\, f(m,m,m)=M\ \&\ f(m,m,m)=M\, ) \implies m=M
 $$ 
Then \quad $y_{n+1} = f(y_n,y_{n-1},y_{n-2})$ 
 has a unique equilibrium in $ [a,b]$,   
 and every solution with initial values in $[a,b]$ converges to the equilibrium.
 \end{theorem}

\begin{theorem}[\cite{Camouzis Ladas}]
\label{th: Camouzis-Ladas}
Let $I$ be a set of real numbers and let $F:I\times I \rightarrow I$ be a   function
$F(u,v)$ which decreases in $u$ and increases in $v$.
Then for every solution $\{x_n\}_{n=-1}^\infty$ of the equation
\begin{equation}
\label{eq: Camouzis-ladas}
x_{n+1} = F(x_n,x_{n-1}),\quad n=0,1,\ldots
\end{equation}
the subsequences $\{x_{2n}\}$ and $\{x_{2n+1}\}$ of even and odd terms 
do exactly one of the following:
\begin{itemize}
\item[\rm (i)] They are both monotonically increasing.
\item[\rm (ii)] They are both monotonically decreasing.
\item[\rm (iii)]  Eventually, one of them is monotonically increasing 
and the other is monotonically decreasing.
\end{itemize}
\end{theorem}
Theorem \ref{th: Camouzis-Ladas} has this corollary.
\begin{corollary}[\cite{Camouzis Ladas}]
If $I$ is a compact interval, then every solution of Eq.(\ref{eq: Camouzis-ladas})
converges to an equilibrium or to a prime period-two solution.
\end{corollary}

\begin{theorem}[\cite{Kocic Ladas}]
\label{th: Kocic-Ladas}
Assume the following conditions hold:
\begin{itemize}
\item[\rm (i)] $h \in C[(0,\infty)\times(0,\infty),(0,\infty)]$.
\item[\rm (ii)] $h(x,y)$ is decreasing in $x$ and strictly decreasing in $y$.
\item[\rm (iii)] $x \, h(x,x)$ is strictly increasing in $x$.
\item[\rm (iv)] The equation 
\begin{equation}
\label{eq: kocic-ladas}
x_{n+1} = x_n\,h(x_n,x_{n-1}),\quad n=0,1,\ldots
\end{equation}
has a unique positive equilibrium $\overline{x}$.
\end{itemize}
Then $\overline{x}$ is a global attractor of all positive solutions of Eq.(\ref{eq: kocic-ladas}).
\end{theorem}


\section{Existence of an Invariant And Attracting Interval}
\label{section: Existence}
In this section we prove a proposition which is key for later developments.  
We will need the function 
\begin{equation}
\label{eqn: function f}
   f\,(\,x\,,\,y \,)\,:=\,\displaystyle \frac{r+p\, x+y }{q\, x+y},
   \quad x,y \in [L,\infty)
\end{equation}
associated to Eq.(3-2-L).
\begin{proposition}
\label{prop: invariant and attracting}
At least one of the following statements is true:
\begin{itemize}
\item[\rm (A)]  
Every solution to (3-2-L) converges to the  equilibrium.
\item[\rm (B)]   
There exist  $m^*$, $M^*$ with \  $L\, < m^*< M^*$ s.t. 
\begin{itemize}
\item[\rm (i)]  $[m^{*},M^{*}]$ is an invariant interval for Eq.(3-2-L), i.e., 
$f([m^{*},M^{*}]\times [m^{*},M^{*}]) \subset [m^{*},M^{*}]$.
\item[\rm (ii)] Every solution  to Eq.(3-2-L) eventually enters $[m^{*},M^{*}]$.
\item[\rm (iii)] $f(x,y)$ is coordinate-wise strictly monotonic on $[m^{*},M^{*}]^{2}$.
\end{itemize}
\end{itemize}
\end{proposition}

The next lemma states that the function $f(\cdot,\cdot)$ associated to Eq.(3-2-L)
is bounded.  
\begin{lemma}
\label{lemma: f is bounded}
There exist {\em positive} constants $\mathcal{L}$ and 
$\mathcal{U}$ such that $L < \mathcal{L}$ and 
\begin{equation}
\label{ineq: bounds for f}
\mathcal{L} \leq f(x,y) \leq \mathcal{U}, \quad x,\ y \in [L,\infty)
\end{equation}
In particular,  
\begin{equation}
\label{eq: [u,l] is invariant for f}
f\left([\mathcal{L},\mathcal{U}]\times[\mathcal{L},\mathcal{U}]\right) \subset [\mathcal{L},\mathcal{U}]\end{equation}
\end{lemma}
\begin{proof}
The function 
$$
\tilde{f}(x,y) = \frac{\alpha + \beta\, x + \gamma\, y}{A+ B\, x + C\, y},
\quad (x,y) \in (0,\infty)^2
$$
associated to Eq.(3-3) is bounded:
$$
\frac{\min\{\alpha,\beta,\gamma\}}{\max\{A,B,C\}}
\leq
\frac{\alpha + \beta\, x + \gamma\, y}{A+ B\, x + C\, y}
\leq
\frac{\max\{\alpha,\beta,\gamma\}}{\min\{A,B,C\}},
\quad (x,y) \in (0,\infty)^2
$$
Set $\tilde{\mathcal{L}} := \frac{\min\{\alpha,\beta,\gamma\}}{\max\{A,B,C\}}$ and 
$\tilde{\mathcal{U}} := \frac{\max\{\alpha,\beta,\gamma\}}{\min\{A,B,C\}}$.
The affine change of coordinates (\ref{change of coordinates}) maps the
 rectangular region $[\tilde{\mathcal{L}},\tilde{\mathcal{U}}]^2$ onto a rectangular region
 $[\mathcal{L},\mathcal{U}]^2$ which satisfies (\ref{ineq: bounds for f}) and (\ref{eq: [u,l] is invariant for f}).
\end{proof}
\begin{lemma}
\label{lemma: p=q}
If $p=q$, then every solution to  Eq.(3-2-L)  converges to the 
unique equilibrium.
\end{lemma}
\begin{proof}
If $p=q$ then $D_1 f(x,y) =  - \frac{p\,r}
     {{\left( p\,x + y \right)}^2}  $ and
  $D_2 f(x,y) = - \frac{r}
     {{\left( p\,x + y \right) }^2}  $.
     Thus, depending on the sign  of $r$, 
     the function $f(x,y)$ is either  nondecreasing in both coordinates, or 
      nonincreasing in both coordinates on  $[L,\infty)$.  By Lemma \ref{lemma: f is bounded}, all solutions 
      $\{y_n\}_{n=-1}^\infty$ satisfy $y_n \in [\mathcal{L},\mathcal{U}]$ for $n\geq 1$.
      A direct algebraic calculation may be used to show that
all solutions $(m,M)\in [\mathcal{L},\mathcal{U}]$ of either one of the systems
of equations
$$
\left\{
\begin{array}{rcl}
 M &=& f(M, M) \\ 
 m &=& f(m, m)
\end{array}
\right.
\quad \mbox{and} \quad 
\left\{
\begin{array}{rcl}
 M &=& f(m, m) \\
  m &= & f(M, M)
\end{array}
\right.
$$
necessarily satisfy  $m=M$.
     In either case, the hypotheses (i) or (ii) of Theorem \ref{th: mM first} are satisfied,
     and the conclusion of the lemma follows.
\end{proof}

We will need the following elementary result, which is given here  without proof.
\begin{lemma}
\label{lemma: signs of partials}
Suppose $q\neq p$.
The function $f(x,y)$ has continuous partial derivatives on $(L, \infty)^2$, and 
\begin{itemize}
\item[\rm i.]
$D_1 f(x,y) = 0$ if and only if   $y=\frac{q\,r}{q-p}$, and
$D_1 f(x,y) > 0$ if and only if   $(p-q)\,y > q\,r$.
\item[\rm ii.]
$D_2 f(x,y) = 0$ if and only if   $x=\frac{-r}{p-q}$, and
$D_2 f(x,y) > 0$ if and only if   $(q-p)\,x> r$.
\end{itemize}
\end{lemma}
We will need to refer to the values $K_1$ and $K_2$ where
the partial derivatives of $f(x,y)$ change sign.
\begin{definition}
If $p \neq q$, set
$$
K_1:=\frac{q\,r}{p-q} \quad \mbox{and} \quad K_2:=\frac{-r}{p-q}
$$
\end{definition}
\begin{definition}
For $L \leq m \leq M$, let
$$
\phi(m,M) := \min \{ f(x,y) : (x,y) \in [m,M]^2\}
\quad \mbox{and} \quad 
\Phi(m,M) := \max \{ f(x,y) : (x,y) \in [m,M]^2\}
$$
\end{definition}
\begin{lemma} 
\label{lemma: m < phi}
Suppose $p \neq q$.
If $[m,M]\subset [\mathcal{L},\mathcal{U}]$ is an invariant interval for Eq.(3-2-L) 
with $m \leq K_1 \leq M$ or $m \leq K_2 \leq M$,  
then  $m < \phi(m,M)$ or $\Phi(m,M) < M$ or $m=M=\overline{y}$.
\end{lemma}
\begin{proof}
By definition of $\phi$ and $\Phi$,  $m \leq \phi(m,M)$ and $\Phi(m,M) \leq M$.
Suppose 
\begin{equation}
\label{eq: m = phi}
m = \phi(m,M)\quad \mbox{and} \quad \Phi(m,M) = M
\end{equation} 
The proof will be complete when it is shown that $m=M$.
There are a total of four cases to consider:
(a) $r \geq 0$ and $p>q$, (b) $r<0$ and $p<q$,
(c) $r \geq 0$ and $p<q$, and (d) $r<0$ and $p>q$.
We present the proof of case (a) only, as the proof of the other cases
is similar.

If $r \geq 0$ and $p>q$, then $K_1 \in [m,M]$ and $K_2 \not \in [m,M]$.
Note that 
$$
[m,M]\times[m,M] = [m,M]\times [m,K_1]\ \bigcup \ [m,M]\times [K_1,M].
$$
By Lemma \ref{lemma: signs of partials}, the signs of the partial derivatives of $f(x,y)$ are constant on 
the interior of each of the sets 
$[m,M]\times [m,K_1] $ and $[m,M]\times [K_1,M]$, as shown in the diagram.
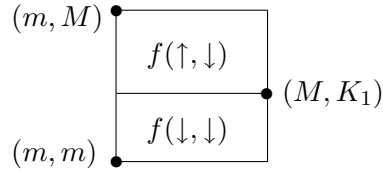
\begin{figure}[!h]
\begin{center}
\setlength{\unitlength}{1cm} 
 \begin{picture}(2.5,2.5)
 \put(0.5,2.5){\textbullet}
  \put(2.5,1.4){\textbullet}
   \put(0.5,0.5){\textbullet}
 \put(-0.8,.6){$(m,m)$}
 \put(2.8,1.4){$(M, K_{1})$}
 \put(-0.8,2.4){$(m,M)$}
 \put(0.6, 0.6){\framebox(2,2){}}
 \put(1,0.9){$f(\downarrow,\downarrow)$}
 \put(0.6,1.5){\line(1,0){2}}
 \put(1,1.9){$f(\uparrow,\downarrow)$}
 \end{picture}
\end{center} 
\label{fig: diagram 1}
\caption{The arrows indicate type of coordinate-wise monotonicity of $f(x,y)$ on each region.}
\end{figure}

 Since $f(x,y)$ is nonincreasing in both $x$ and $y$ on $[m, M]\times[m, K_{1}]$, 
\begin {equation}
\label{eq : ineq1-new1}
f(M, K_{1}) \,\leq\, f(x, y) \,\leq\, f(m, m)\ \mbox{ for }\ (x, y)\in [m, M]\times[m, K_{1}].
\end{equation}
Similarly,  $f(x,y)$ is nondecreasing in $x$ and nonincreasing in $y$ on $[m, M]\times[K_{1}, M]$, hence
\begin {equation}
\label{eq : ineq2-new1}
f(m, M) \,\leq\, f(x, y) \,\leq\, f(M, K_{1})\ \mbox{ for }\ (x, y)\in [m, M]\times[K_{1}, M].
\end{equation}
From (\ref{eq : ineq1-new1}) and (\ref{eq : ineq2-new1}) one has
\begin{equation}
\label{eq: phi = f(m,M)}
\phi(m,M) = f(m,M) \quad \mbox{and} \quad \Phi(m,M)=f(m,m)
\end{equation}
Combine (\ref{eq: phi = f(m,M)}) with relation (\ref{eq: m = phi}) to obtain the system of equations
\begin{equation}
\label{eq: sys phi Phi}
\left\{ \begin{array}{rcl}
f(m,M) & = & m \\ \\
f(m,m) & = & M
\end{array}
\right.
\end{equation} 
Eliminating $M$ from system (\ref{eq: sys phi Phi}) gives the cubic in $m$
\begin{equation}
\label{eq: cubic in m}
q\,(q+1)\,m^3 + (1-p\,q) \, m^2 + (-1-p-q\,r)\,m-r = 0
\end{equation}
which has the roots
\begin{equation}
\label{eq: the m roots}
- \frac{1}{q} , \quad 
\frac{1-p-\sqrt{(1+p)^2+4\,r\,(1+q)}}{2\,(1+q)}, \quad \mbox{and} \quad \
\frac{1-p+\sqrt{(1+p)^2+4\,r\,(1+q)}}{2\,(1+q)}
\end{equation}
Only one root in the list  (\ref{eq: the m roots}) is positive, namely
$$
m = \frac{1-p+\sqrt{(1+p)^2+4\,r\,(1+q)}}{2\,(1+q)} = \overline{y}
$$
Substituting into one of the equations of system (\ref{eq: sys phi Phi})
one also obtains $M=\overline{y}$, which gives the desired relation $m=M=\overline{y}$.
\end{proof}
\begin{definition}
Let $m_0:=\mathcal{L}$, $M_0:=\mathcal{U}$, and for $\ell=0,1,2,\ldots$ let
$m_\ell := \phi(m_\ell,M_\ell)$, $M_\ell :=\Phi(m_\ell,M_\ell)$.
\end{definition}
%
%
By the definitions of $m_\ell$, $M_\ell$, $\phi(\cdot,\cdot)$ and $\Phi(\cdot,\cdot)$,
we have that 
$[m_{\ell+1},M_{\ell+1}] \subset [m_\ell,M_\ell]$
for $\ell =0,1,2,\ldots$.
Thus the sequence $\{m_\ell\}$ is nondecreasing and $\{M_\ell\}$ is nonincreasing.
Let  $m^*:=\lim m_\ell$ and $M^*:=\lim M_\ell$.
\begin{lemma} Suppose $p \neq q$.
Either there exists $N \in \mathbb{N}$ such that 
$\{K_1,K_2\} \cap [m_N,M_N] = \emptyset$, or 
$m^* = M^* = \overline{y}$.
\end{lemma}
\begin{proof}
Arguing by contradiction, suppose $m^* < M^*$ and  
for all $\ell\in \mathbb{N}$, $\{K_1,K_2\} \cap [m_\ell,M_\ell] \not = \emptyset$.
Since the intervals $[m_\ell,M_\ell]$ are nested and 
$\cap[m_\ell,M_\ell] = [m^*,M^*]$, it follows that 
$\{K_1,K_2\} \cap [m^*,M^*] \not = \emptyset$.
By Lemma \ref{lemma: m < phi}, we have 
\begin{equation}
\label{ineq: mstar1}
m^* < \phi(m^*,M^*) \quad \mbox{or} \quad \Phi(m^*,M^*)<M^*
\end{equation}
Continuity of the functions $\phi$ and $\Phi$
implies
\begin{equation}
\label{ineq: mstar2}
\phi(m^*,M^*) = \lim \phi(m_\ell,M_\ell) = \lim m_{\ell+1} = m^* 
\quad \mbox{or} \quad 
\Phi(m^*,M^*) = \lim \Phi(m_\ell,M_\ell) = \lim M_{\ell+1} = M^* 
\end{equation}
Statements (\ref{ineq: mstar1}) and (\ref{ineq: mstar2}) give a contradiction.
\end{proof}
\bigskip

\noindent {\bf Proof of Proposition \ref{prop: invariant and attracting}}. 
Suppose statement (A) is not true.
By Lemma \ref{lemma: p=q}, one must have $p\neq q$.
Note that if $\{y_\ell\}$ is a solution to Eq.(3-2-L), then
$y_{\ell+1} \in [m_\ell, M_\ell]$ for $\ell=0,1,2,\ldots$.
If $m^*=M^*$, since $m_\ell \rightarrow m^*$ and $M_\ell\rightarrow M^*$
we have $y_\ell \rightarrow \overline{y}$, but this is statement (A) which we 
are negating.  Thus 
 $m^* < M^*$, and by Lemma \ref{lemma: m < phi} there exists $N \in \mathbb{N}$ such that
$\{K_1,K_2\} \cap [m_N,M_N]=\emptyset$, so $f(x,y)$ is coordinate-wise monotonic
on $[m_N,M_N]$.
The set $[m_N,M_N]$ is invariant, and every solution enters $[m_N,M_N]$ starting at least 
with the term with subindex $N+1$.  
We have shown that if statement (A) is not true, then statement (B) is necessarily true.
This completes the proof of the proposition.
\hfill $\Box$
\bigskip


\section{The Equation (3-2) with $r\geq 0$, $p>q$ and $\frac{q\,r}{p-q} < \frac{p}{q}$}
\label{sec: The Equation (3-2)}
In this section we restrict our attention to the equation 
 \begin{equation*}
 \tag{3-2}
 y_{n+1} = f(x_n, x_{n-1})
 \quad n=0,1,\ldots,\quad  x_{-1},x_0 \in (0,\infty)
 \end{equation*}
 where
 \begin{equation*}
   f\,(\,x\,,\,y \,)\,:=\,\displaystyle \frac{r+p\, x+y }{q\, x+y},
\end{equation*}
For  $p>0$, $q>0$, and $r\geq 0$,
Eq.(3-2) has a unique {\it positive} equilibrium 
\begin{equation}
\label{eq: ybar formula}
\overline{y} = \frac{p+1+\sqrt{(p+1)^2+4 r\, (q+1)}}{2\,(q+1)}
\end{equation}
We note that if $I\subset (0,\infty)$ is an invariant compact interval, 
then necessarily $\overline{y} \in I$.

The goal in this section is to prove the following proposition,
which will provide an important part of the proofs of Theorems
\ref{th: 3-3} and \ref{th: 3-2}.
\begin{proposition}
\label{prop: f up down r >0 on I}
Let $p$, $q$ and $r$  be real numbers  such that 
\begin{equation}
\label{eq: hypotheses 3-2}
p>q>0, \quad r \geq 0,\quad \mbox{\rm and} \quad \frac{q\,r}{p-q} < \frac{p}{q}
\end{equation}
and let $[\tilde{m},\tilde{M}]\subset (\frac{q\,r}{p-q} , \frac{p}{q})$  
be a compact invariant interval for Eq.{\rm (3-2)}.   
Then every solution to Eq.{\rm (3-2)} with $x_{-1}, x_0 \in [\tilde{m},\tilde{M}]$ 
converges to the equilibrium.
\end{proposition}
Proposition \ref{prop: f up down r >0 on I} follows from 
Lemmas \ref{lemma: f up down ybar p<1 q>1 r>0}, 
\ref{lemma: f up down ybar sufficient condition} and 
\ref{lemma: ybar is GA when suff cond 2},
which are stated and proved next.
\bigskip

%
\begin{lemma}
\label{lemma: f up down ybar p<1 q>1 r>0}
Assume the hypotheses to Proposition \ref{prop: f up down r >0 on I}.
If either $q \geq 1$ or $p \leq 1$, then 
     every solution to Eq.{\rm (3-2)} with $x_{-1}, x_0 \in [\tilde{m},\tilde{M}]$ 
converges to the equilibrium.
\end{lemma}
\begin{proof}
We verify  that hypothesis (iv) of  Theorem \ref{th: mM first} is true.
Since $x > \frac{r\,q}{p-q}$ for $x \in [\tilde{m},\tilde{M}]$, 
the function $f(x,y)$ is increasing in $x$ and decreasing in $y$ for 
$(x,y) \in [\tilde{m},\tilde{M}]^2$ by Lemma \ref{lemma: signs of partials}.
Let $m, \, M \in [\tilde{m},\tilde{M}]$ be such that $m \neq M$ and 
\begin{equation}
\label{eq: Mm for p<1}
\left\{
\begin{array}{rcl}
f(M,m) - M & = & 0\\
f(m,M) - m & = & 0
\end{array}
\right.
\end{equation}
We show first that system (\ref{eq: Mm for p<1}) 
has no solutions if either $q \geq 1$ or $p \leq 1$.
By eliminating denominators in both equations in  (\ref{eq: Mm for p<1}),
\begin{equation}
\label{eq: Mm for p<1 b}
\left\{
\begin{array}{rcl}
m - m\,M + M\,p - M^2\,q + r & = & 0 \\
M - m\,M + m\,p - m^2\,q + r & = & 0
\end{array}
\right.
\end{equation}
and by subtracting terms in (\ref{eq: Mm for p<1 b})  one obtains 
\begin{equation}
\label{eq: Mm for p<1 2}
(M-m)(  1-p + q\,(m+M) ) = 0
\end{equation}
Since $m\neq M$, we have $q \,(m+M) = p -1$, which  implies 
that for $p\leq 1$ there are no solutions to  system (\ref{eq: Mm for p<1})
which have both coordinates positive.
Now assume $p >1$;  from (\ref{eq: Mm for p<1 2}), 
$m = \frac{p-1-M\,q}{q}$, and substitute the latter into 
(\ref{eq: Mm for p<1}) to see that $x=M$ is a solution to the quadratic equation
\begin{equation}
\label{eq: Mm quadratic p>1}
\left( 1 - q \right)\, x^2  + 
  \frac{\left( -1 + p\right) \,
     \left( -1 + q \right) }{q}\, x + \frac{-1 + p + q\,r}{q} = 0
\end{equation}
By a symmetry argument, one has that $x=m$ is also a solution to (\ref{eq: Mm quadratic p>1}).
By inspection of the coefficients of the polynomial in the left-hand-side of 
(\ref{eq: Mm quadratic p>1}) one sees that two positive solutions are possible only when $q<1$. 
To get the conclusion of the lemma, note that the fact that 
(\ref{eq: Mm for p<1}) has no solutions with $m\neq M$
is just hypothesis (iv) of  Theorem \ref{th: mM first}.
\end{proof}
\begin{lemma}
\label{lemma: f up down ybar sufficient condition}
Assume the hypotheses to Proposition \ref{prop: f up down r >0 on I}.
If 
\begin{equation}
\label{sufficient condition}
r\ \leq\ p^2\, q - p
\end{equation}
then 
     every solution to Eq.{\rm (3-2)} with $x_{-1}, x_0 \in [\tilde{m},\tilde{M}]$ 
converges to the equilibrium.
\end{lemma}
\begin{proof}
By substituting $x_n = f(x_{n-1},x_{n-2})$ into $x_{n+1} = f(x_{n},x_{n-1})$ we obtain
\begin{equation}
\label{eq: embed}
x_{n+1} = \frac{r + p\, x_n + x_{n-1}}{q\, x_n + x_{n-1}}
= 
\frac{r + p\, \frac{r + p\, x_{n-1} + x_{n-2}}{q\, x_{n-1} + x_{n-2}} + x_{n-1}}
{q\, \frac{r + p\, x_{n-1} + x_{n-2}}{q\, x_{n-1} + x_{n-2}} + x_{n-1}}, \quad n=0,1,\ldots
\end{equation}
that is,
\begin{equation}
\label{eq: embed 2}
x_{n+1} = \hat{f}(x_n,x_{n-1},x_{n-2}), \quad \mbox{\rm where }
\hat{f}(x,y,z) = \frac{p\,r + p^2\,y + q\,r\,y + 
    q\,y^2 + p\,z + r\,z + y\,z}
    {q\,r + p\,q\,y + q\,x^2 + 
    q\,z + y\,z}
\end{equation}
where the $x$ has been kept in $\hat{f}(x,y,z)$ for bookkeeping purposes.
Thus $\hat{f}(x,y,z)$ is constant in $x$.  
We claim $\hat{f}(x,y,z)$ is decreasing in both
$y$ and $z$.  To see that the partial derivative 
\begin{equation}
\label{eq: embed partial 3}
D_3\hat{f}(x,y,z) = -  \frac{\left( r + (p-q)\,y \right) \,
      \left( - q\,r   + (p-q)\,y \right) }
      {{\left( q\,r + p\,q\,y + 
         q\,y^2 + q\,z + y\,z
         \right) }^2} 
\end{equation}
is negative just use $p>q$ and 
the inequality $(p-q)\, y -q\,r >0$, which is true by Lemma \ref{lemma: signs of partials}.
The remaining partial derivative is
\begin{equation}
\label{eq: embed partial 2}
D_2\hat{f}(x,y,z) = - \frac{L(y,z)}
    {{\left( q\,r + p\,q\,y + 
         q\,y^2 + q\,z + y\,z
         \right) }^2} 
\end{equation}
where
\begin{eqnarray*}
h(y,z) := &
- q^2\,r^2   + 
  2\,p\,q\,r\,y - 2\,q^2\,r\,y + 
  p^2\,q\,y^2 - p\,q^2\,y^2 + 
  q^2\,r\,y^2 + p\,r\,z - q\,r\,z + 
  \\ & \quad + \ 
  p\,q\,r\,z - q^2\,r\,z + 
  2\,p\,q\,y\,z - 2\,q^2\,y\,z + 
  2\,q\,r\,y\,z + p\,z^2 - q\,z^2 + 
  r\,z^2
\end{eqnarray*}
We have,
\begin{eqnarray}
D_1h(y,z) &=& 2\,\left( p - q \right) \,q\,r + 
  2\,q\,\left( p^2 - p\,q + 
     q\,r \right) \,y + 
  2\,q\,\left( p - q + r \right) \,z > 0,
  \\
  D_2h(y,z) &=& 
\left( p - q \right) \,
   \left( 1 + q \right) \,r + 
  2\,q\,\left( p - q + r \right) \,
   y + 2\,\left( p - q + r \right)
     \,z > 0
\end{eqnarray}
Since $\frac{q \,r}{p-q} \leq \tilde{m}$,    
$$
h(y,z) \geq 
h\left(\frac{q\,r}{p-q},\frac{q\,r}{p-q}\right) =
\frac{q\,{\left( 1 + q \right) }^2\,
    r^2\,\left( p^2 - p\,q + q\,r \right) }{{\left( p - q \right) }^2} > 0\, ,
    \quad y,\, z \, \in [\tilde{m},\tilde{M}]
$$
thus we conclude that $D_2\hat{f}(x,y,z) <0$ for $x,y,z \in [\tilde{m},\tilde{M}]$.

To complete the proof we verify the hypotheses of Theorem \ref{th: Mm 3-3}.
We claim that the system of equations
\begin{equation}
\label{eq: Mm for embed}
\left\{
\begin{array}{rcl}
\hat{f}(m,m,m) - M & = & 0 \\
\hat{f}(M,M,M) - m & = & 0
\end{array}
\right.
\end{equation}
has no solutions $(m,M)$ with $m\neq M$ whenever hypothesis (\ref{sufficient condition}) holds.

By eliminating denominators in both equations in  (\ref{eq: Mm for embed}) one obtains
{\small
\begin{equation}
\label{eq: Mm for p<1 b embed}
\begin{array}{rcl}
-m^2 + m^2\,M - m\,p - m\,p^2 - 
   m^2\,q + m\,M\,q + m^2\,M\,q + 
   m\,M\,p\,q - m\,r - p\,r - 
   m\,q\,r + M\,q\,r & = & 0 \\
  -M^2 + m\,M^2 - M\,p - M\,p^2 + 
   m\,M\,q - M^2\,q + m\,M^2\,q + 
   m\,M\,p\,q - M\,r - p\,r + 
   m\,q\,r - M\,q\,r & = & 0
\end{array}
\end{equation}
} 
and by subtracting terms in (\ref{eq: Mm for p<1 b embed})  one obtains 
\begin{equation}
\label{eq: Mm for p<1 2 embed}
\left( \left( m - M \right) \,
    \left( -m - M + m\,M - p - 
      p^2 - m\,q - M\,q + m\,M\,q - 
      r - 2\,q\,r \right)  \right) = 0
\end{equation}
Since $m\neq M$, we may use the second factor in the left-hand-side term of (\ref{eq: Mm for p<1 2 embed})
to solve for $M$ in terms of $m$, which upon substitution into $\hat{f}(m,m,m)=M$ 
and simplification yields the equation
\begin{equation}
\label{eq: no solutions}
\frac{a_2\, m^2 + a_1\, m + a_0}
  {\left( -1 + m \right) \,
    \left( 1 + q \right) \,
    \left( m^2 + m\,q + m^2\,q + 
      m\,p\,q + q\,r \right) } = 0
\end{equation}
where
\begin{eqnarray*}
a_0 & = & r\,\left( p + 2\,p\,q + p^2\,q + 
     q\,r + 2\,q^2\,r \right) \\
a_1 & = & p + p^2 + 2\,p\,q + 3\,p^2\,q + 
   p^3\,q + r - p\,r + 4\,q\,r + 
   4\,q^2\,r + 2\,p\,q^2\,r \\
a_2 & = & \left( 1 + q \right) \,
   \left( 1 + 2\,q + p\,q + 
     q\,r \right)
\end{eqnarray*}
By hypothesis (\ref{sufficient condition}) 
we have $r\, p \leq p^3\,q - p^2 < p^3 q$, hence $p^3\,q - r \, p >0$, which implies $a_1 \geq 0$.
By direct inspection one can see that $a_0>0$ and $a_2>0$.
 Thus (\ref{eq: no solutions}) has no positive solutions, and we conclude that 
(\ref{eq: Mm for embed}) has no solutions $(m,M) \in [\tilde{m},\tilde{M}]$ with $m\neq M$.
We have verified the hypotheses of Theorem \ref{th: Mm 3-3}, and the conclusion
of the lemma follows.
\end{proof}


 \begin{lemma}
 \label{lemma: LAS}
 Let $p>0$, $q>0$ and $r\geq 0$.
 If the positive equilibrium  $\overline{y}$  of Eq.(3-2) satisfies 
 $\overline{y}<\frac{p}{q}$, then $\overline{y}$ 
  is locally asymptotically stable {\rm (L.A.S.)}.
  \end{lemma}
 \begin{proof}
 Solving for $r$ in 
 $$
 \overline{y} = \frac{r+(p+1)\, \overline{y}}{(q+1)\,\overline{y}}
 $$
 gives
 \begin{equation}
 \label{eq: r in terms of ybar}
 r = (q+1) \,\overline{y}^2 - (p+1)\,\overline{y}
\end{equation}
Then a calculation shows 
$$
\begin{array}{rcl}
D_1 f(\overline{y},\overline{y}) &=& \frac{p-q\,\overline{y}}{\overline{y}\,(q+1)} \\ \\
D_2 f(\overline{y},\overline{y}) &=& - \frac{\overline{y}-1}{\overline{y}\,(q+1)} 
 \end{array}
$$
Set $t_1 := D_1 f(\overline{y},\overline{y})$ and $t_2 :=D_2 f(\overline{y},\overline{y})$.
The equilibrium $\overline{y}$ is locally asymptotically stable if the roots of the 
characteristic polynomial 
$$
\rho(x) = x^2 - t_1\, x - t_2
$$
have modulus less than one \cite{KuL01}. 
 By the Schur-Cohn Theorem, $\overline{y}$ is L.A.S.
 if and only if $|t_1| < 1 - t_2 < 2$.
 It can be easily verified that $1- t_2<2$ if and only if 
 $0 < q \, \overline{y}+1$ which is true regardless of the 
 allowable parameter values.
 Since $p-q\, \overline{y} > 0$ by the hypothesis, we have
 $|t_1| = |\frac{p-q\,\overline{y}}{\overline{y}(q+1)}| = \frac{p-q\,\overline{y}}{\overline{y}(q+1)}$, hence
 some algebra gives $|t_1| < 1-t_2$ if and only if
\begin{equation}
\label{eq: schur conn}
  \frac{1}{2} \frac{p+1}{q+1} < \overline{y}
\end{equation}
But (\ref{eq: schur conn}) is a true statement 
by  formula (\ref{eq: ybar formula}).
We conclude $\overline{y}$ is L.A.S.
 \end{proof}
\medskip

\begin{lemma}
\label{lemma: ybar is GA when suff cond 2}
Assume the hypotheses to Proposition \ref{prop: f up down r >0 on I}.
If 
\begin{equation}
\label{sufficient condition 2}
p>1,\quad q<1, \quad \mbox{\rm and} \quad r\ >\ p^2\, q - p
\end{equation}
then  every solution to Eq.{\rm (3-2)} with $x_{-1}, x_0 \in [\tilde{m},\tilde{M}]$ 
converges to the equilibrium.
\end{lemma}
\begin{proof}
The proof begins with  a change of variable in Eq.(3-2)
to produce a transformed equation with normalized coefficients
analogous to those in the standard {\it normalized Lyness' Equation}
\cite{Ladas95}, \cite{Zeeman}, \cite{Kulenovic} 
\begin{equation}
\label{eq: Lyness}
z_{n+1} = \frac{\tilde{\alpha}+x_n}{x_{n-1}}
\end{equation}
We seek to use an argument of proof similar to the one used in \cite{Y2K},
in which one takes advantage of the existence of {\it invariant curves} of Lyness' Equation
to produce a Lyapunov-like function for Eq.(3-2).

Set $y_n = p z_n$ in Eq.(3-2) to obain the equation
\begin{equation}
\label{3-2-z}
z_{n+1} = \frac{a + z_n+g\, z_{n-1}}{b \, z_n+z_{n-1}},\quad n=0,1,\ldots
\quad z_{-1}, z_0 \in (0,\infty)
\end{equation}
where
\begin{equation}
\label{eq: new parameters}
a = \frac{r}{p^2}, \quad g = \frac{1}{p}, \quad b=q
\end{equation}
We shall denote with $\overline{z}$ the unique equilibrium of Eq.(\ref{3-2-z}).
Note that 
\begin{equation}
\label{eq: zbar = p ybar}
\overline{z} = p \, \overline{y}
\end{equation}
It is convenient to parametrize Eq.(\ref{3-2-z}) in terms of the equilibrium.
We will use the symbol $u$ to represent 
the equilibrium $\overline{z}$ of Eq.(\ref{3-2-z}).
By direct substitution of the equilibrium $u=\overline{z}$ into Eq.(\ref{3-2-z}) we obtain
\begin{equation}
\label{eq: alpha}
r = (b+1)\, u^2 - (g+1)\,u
\end{equation}
By (\ref{eq: alpha}), $r \geq 0$ iff   $u \geq \frac{g+1}{b+1}$.
Using (\ref{eq: alpha}) to eliminate $r$ from Eq.(\ref{3-2-z}) gives the following equation
for $b>0$, $g>0$ and $u\geq\frac{g+1}{b+1}$, equivalent to Eq.(\ref{3-2-z}):
\begin{equation}
\label{eq: y2k2u}
z_{n+1} = \frac{(b+1)\, u^2 - (g+1)\,u+ z_n + g\, z_{n-1}}{b\,z_n+z_{n-1}}, \quad n=0,1,2,\ldots,
\quad y_{-1},\, y_0 \in (0,\infty)
\end{equation}
Therefore it suffices to 
prove that all solutions of Eq.(\ref{eq: y2k2u})
converge to the equilibrium $u$.

The following statement is crucial for the proof of the proposition.
\begin{claim}
\label{claim: ybar > 1}
$u > 1$ if and only if $r > p^2\,q - p$.
\end{claim}
\begin{proof}
Since $\overline{y} = p\, \overline{z} = p\, u$, 
we have $u>1$ if and only if $\overline{y} > p$, which holds if and only if 
$$
\frac{p+1+\sqrt{(p+1)^2+4\,r\,(q+1)}}{2\,(q+1)} > p
$$
After an elementary simplification, the latter inequality 
can be rewritten as $r > p^2\,q - p$.
\end{proof}

By the hypotheses of the lemma, 
by Claim \ref{claim: ybar > 1}, and by 
(\ref{eq: new parameters}) and (\ref{eq: alpha}) we have
\begin{equation}
\label{eq: ineqs u,g,b}
b< 1, \quad g < 1, \quad 1 < u < \frac{1}{b}, 
\quad \mbox{\rm and}
\quad \frac{g+1}{b+1} \leq u
\end{equation}

We now introduce a function which is the invariant function for
(\ref{eq: Lyness}) with constant $\tilde{\alpha} = u^2-u$ 
(in this case the the equilibrium of (\ref{eq: Lyness}) is $u$):
\begin{equation}
\label{eq: invariantu}
g(x,y) = \frac{(1+x)(1+y)(u^2-u+x+y)}{x\, y}
\end{equation}
Note that $g(x,y) > 0$ for all $x,y\in(0,\infty)$ whenever $u>1$.
By using elementary calculus, one can show that 
the function $g(x,y)$  has a strict global minimum at $(u,u)$
\cite{Kulenovic}, \cite{Zeeman}, i.e.,
\begin{equation}
\label{eq: Lyness min}
g(u,u) < g(x,y), \quad (x,y) \in (0,\infty)^2
\end{equation}
We need some elementary properties of the sublevel sets
$$
S(c) := \{ (s,t)\in (0,\infty) : g(s,t) \leq c \}\, , \quad c>0
$$
We denote with $Q_\ell(u,u)$, $\ell=1,2,3,4$ the four regions
$$
\begin{array}{rcl}
Q_1(u,u) & := & \{ (x,y) \in (0,\infty)\times(0,\infty)\, : \ u \leq x,\ u\leq y \ \} \\  
Q_2(u,u) & := & \{ (x,y) \in (0,\infty)\times(0,\infty)\, : \ x \leq u,\ u\leq y \ \} \\ 
Q_3(u,u) & := & \{ (x,y) \in (0,\infty)\times(0,\infty)\, : \ x \leq u,\ y\leq u \ \} \\ 
Q_4(u,u) & := & \{ (x,y) \in (0,\infty)\times(0,\infty)\, : \ u \leq x,\ y\leq u \ \} 
\end{array}
$$
Let 
\begin{equation}
T(x,y) := \left(\, y\, ,\, \frac{(g+1)\,u^2-(b+1)\,u + y + g\,x}{b\,y+ x}\, \right), \quad (x,y) \in (0,\infty)\times(0,\infty)
\end{equation}
be the map associated to Eq. (\ref{eq: y2k2u})
(see \cite{KUM}).
\begin{claim}
\label{claim: Q2Q4}
If $(x,y) \in Q_2(u,u) \cup Q_4(u,u) \setminus \{ (u,u) \}$, 
then $g(T(x,y)) < g(x,y)$.
\end{claim}
\begin{proof}
Set
\begin{equation}
\Delta_1(x,y) := g(x,y)-g(T(x,y)) 
\end{equation}
A calculation yields
\begin{equation}
\label{eq: Delta1}
\Delta_1(x,y)  =
-  \frac{\left( 1 + x \right) \, F_1(x,y)\, F_2(x,y) }
      {x\,y\,\left( b\,x + y \right)\,F_3(x,y) }
\end{equation}
where
\begin{eqnarray*}  
\label{eq: F for Delta1}
F_1(x,y)  & := & b\,(x-u)\,(y-\mbox{$\frac{1}{b}$}) + (y-u)\,(b\,u+y+u-g)    \\
F_2(x,y)  & := &  b\,(x-u)^2+b\,(x-u)\,u+b\,(x-u)\,u^2+(u-y)(b\,u^2+y\,g)   \\
F_3(x,y)  & := &  (b+1)\,u^2-(1+g)\,u+x+g \, y  
\end{eqnarray*}
By (\ref{eq: ineqs u,g,b}), for $(x,y) \in Q_4(u,u)\setminus \{ (u,u)\}$ we have $u \leq x$ and 
$y \leq u < \frac{1}{b}$ with $(x,y) \neq (u,u)$, therefore $F_1(x,y) <0$, $F_2(x,y) >0$ and $F_3(x,y)>0$.
Consequently $\Delta_1(x,y)>0$ for $(x,y) \in Q_4(u,u)\setminus \{ (u,u)\}$.
To see that $\Delta_1(x,y)>0$ for $(x,y) \in Q_2(u,u)\setminus \{ (u,u)\}$ as well,
rewrite $F_1(x,y)$ and $F_2(x,y)$ as follows:
\begin{eqnarray*}  
\label{eq: F for Delta11}
F_1(x,y)  & =  &   b(u-x)(\mbox{$\frac{1}{b}$}-u)+(y-u)^2+(y-u)\,u+(y-u)\,(u-g)+b\,(y-u)\,x  \\
F_2(x,y)  & = &   b\,(x-u)\,(x+u^2)-b\,(y-u)\,u^2-(y-u)^2\,g-(y-u)\,u\,g \\
\end{eqnarray*}
For $(x,y) \in Q_2(u,u)\setminus \{ (u,u)\}$ we have $x \leq u \leq y$ and $(x,y) \neq (u,u)$.
Thus $F_1(x,y) > 0$, $F_2(x,y) < 0$, and $F_3(x,y) >0$,
which imply $\Delta_1(x,y) > 0$.
\end{proof}
%

%
\begin{claim}
\label{claim: Q1Q3, g>b}
Suppose $g >  b$. 
If $(x,y) \in Q_1(u,u) \cup Q_3(u,u) \setminus \{ (u,u) \}$, 
then $g(T^2(x,y)) < g(x,y)$.
\end{claim}
\begin{proof}
This proof requires extensive use of a computer algebra system
to verify certain inequalities involving rational expressions.
Here we give an outline of the steps, and refer the reader to 
Section \ref{appendix: computer 1} for the details.

Since $b < g < 1 < u < \frac{1}{b}$,
and $\frac{g+1}{b+1}<u$ we may write 
\begin{equation}
\label{eq: subs u,g}
\begin{array}{rcl}
u &=&   \frac{g+1}{b+1} + t \, , \quad t >0\\ \\
g &=& \frac{b+\frac{s}{b}}{1+s} \, , \quad s >0
\end{array}
\end{equation}
The expression $\Delta_2 :=g(x,y)- g(T^2(x,y))$
may be written as a single ratio of polynomials, 
$\Delta_2 = \frac{N}{D}$ with $D > 0$.
The next step is to show $N>0$ for $(x,y)\neq (u,u)$.

Points $(x,y)$ in $Q_1(u,u)$ may be written in the form $x=u+v$, $y=u+w$,
where $v,w\in [0,\infty)$.  
Substituting $x$, $y$, $u$ and $g$ in terms of $v$, $w$, $s$ and $t$
into the expression for $N$ one obtains a rational expression 
$\frac{\tilde{N}}{\tilde{D}}$ with positive denominator. 
The numerator  $\tilde{N}$ has some negative coefficients.
At this points two cases are considered, $w \geq v$,  and $w\leq v$.
These can be written as $w = v+k$ and $v=w+k$ for  nonnegative $k$.
Substitution of each one of the latter expressions in $\tilde{N}$ 
gives a polynomial with positive coefficients.  This proves 
$\Delta_2(x,y)>0$ for $(x,y) \in Q_2(u,u)$.

If now we assume $(x,y) \in Q_3(u,u)$ with $(x,y) \neq (u,u)$,
we may write 
\begin{equation}
\begin{array}{rcl}
x &=& \frac{u}{v+1} \quad v \in [0,\infty) \\ \\
y &=& \frac{u}{w+1} \quad w \in [0,\infty)
\end{array}
\end{equation}
The rest of the proof is as in the first case already discussed.
Details can be found in Section \ref{appendix: computer 1}.
\end{proof}

%
\begin{claim}
\label{claim: Q1Q3, g<b}
Suppose $g <  b$ and $u>1$. 
If $(x,y) \in Q_1(u,u) \cup Q_3(u,u) \setminus \{ (u,u) \}$, 
then $g(T^3(x,y)) < g(x,y)$.
\end{claim}
\begin{proof}
The proof is analogous to the proof of Claim \ref{claim: Q1Q3, g>b}. We provide an outline.
More details can be found in Section \ref{appendix: computer 1}.

Since $1<u$, we may write $u=1+t$ with $t >0$.
Also, $u < \frac{1}{b}$ implies $b < \frac{1}{u}$, and  
$b = \frac{1}{1+t+s}$ for $s >0$.  
Since $g<b$ we may write  
$g = \frac{1}{1+t+s+\ell}$ for $\ell>0$.

The expression $\Delta_3 := g(x,y)-g(T^3(x,y))$
may be written as a single ratio of polynomials, 
$\Delta_2 = \frac{N}{D}$ with $D > 0$.
The next step is to show $N>0$ for $(x,y)\neq (u,u)$.
This is done in a way similar to the procedure described in 
in Claim \ref{claim: Q1Q3, g>b}.
\end{proof}

To complete the proof of the lemma, let   
$(\phi,\psi ) \in (0,\infty)\times(0,\infty)$.
Let $\{y_n\}_{n\geq-1}$ be the solution to (\ref{eq: y2k2u})
with initial condition $(y_{-1},y_0)=(\phi,\psi)$,
and let $\{T^n(\phi,\psi)\}_{n\geq 0}$ be the corresponding orbit of $T$.
The following argument is essentially the same as the one found
in \cite{Y2K}; we provided here for convenience.
Define 
\begin{equation}
\label{eq: hatc def}
\hat{c} := \liminf_n g(T^n(\phi,\psi))  
\end{equation}
Note that $\hat{c} < \infty$, which  can be shown by 
applying Claims \ref{claim: Q2Q4}, \ref{claim: Q1Q3, g>b} and \ref{claim: Q1Q3, g<b} 
repeatedly as needed
to obtain a nonincreasing subsequence of 
$\{g(T^n(\phi,\psi))\}_{n\geq 0}$ that is bounded below by $g(u,u)$.
Let $\{g(T^{n_k}(\phi,\psi))\}_{k\geq 0}$ be a subsequence
convergent to $\hat{c}$.
Therefore there exists $c>0$ such that 
$$
g(T^{n_k}(\phi,\psi)) \leq c \quad \mbox{ for all $k\geq 0$,}
$$
that is, 
$$T^{n_k}(\phi,\psi)) \in S(c) := \{ (s,t) : g(s,t) \leq c \}\quad \mbox{for } k \geq 0
$$
The set $S(c)$ is closed by continuity of $g(x,y)$.
Boundedness of $S(c)$ follows from  
$$0< x,\, y <   
\frac{(1+x)(1+y)(u^2-u+x+y)}{x\, y} = g(x,y) = c ,
\quad \mbox{for} \quad (x,y) \in S(c)
$$
Thus $S(c)$ is compact, and 
there exists a convergent subsequence 
$\{ T^{n_{k_\ell}}(\phi,\psi))\}_\ell$ with limit $(\hat{x},\hat{y})$.
Note that
\begin{equation}
\label{eq: hat c = g(xhat,yhat)}
\hat{c} = \lim_{\ell \rightarrow \infty} g(T^{n_{k_\ell}}(\phi,\psi))
= g(\hat{x},\hat{y})
\end{equation}
We claim that $(\hat{x},\hat{y}) = (u,u)$.
If not, then by claims
\ref{claim: Q2Q4} \ref{claim: Q1Q3, g>b} and \ref{claim: Q1Q3, g<b},
\begin{equation}
\label{eq: cont ineq}
\min \{ g(T(\hat{x},\hat{y})),g(T^2(\hat{x},\hat{y})),g(T^3(\hat{x},\hat{y}))\} < \hat{c}
\end{equation}
Let $\| \cdot \|$ denotes the euclidean norm.
By (\ref{eq: cont ineq}) and continuity, there exists $\delta > 0$ such that 
\begin{equation}
\label{eq: by continuity}
\| (s,t) - (\hat{x},\hat{y}) \| < \delta \implies 
\min \{ g(T(s,t)),g(T^2(s,t)),g(T^3(s,t))\} < \hat{c}
\end{equation}
Choose $L\in \mathbb{N}$ large enough so that 
\begin{equation}
\label{eq: choose L}
\| T^{n_{k_L}}(\phi,\psi) - (\hat{x},\hat{y}) \| < \delta
\end{equation}
But then (\ref{eq: by continuity}) and (\ref{eq: choose L}) imply
\begin{equation}
\min \{ g(T^{n_{k_L}+1}(\phi,\psi)),g(T^{n_{k_L}+2}(s,t)),g(T^{n_{k_L}+3}(s,t))\} < \hat{c}
\end{equation}
which contradicts the definition (\ref{eq: hatc def}) of $\hat{c}$.
We conclude $(\hat{x},\hat{y}) = (u,u)$.
From this and the definition of convergence of sequences we have that for every $\epsilon > 0$ 
there exists $L \in \mathbb{N}$ such that
$\| T^{n_{k_L}}(\phi,\psi) - (u,u) \| < \epsilon$.
Finally, since 
$$
\max \left\{ |y_{n_{k_L}-1} - u | , |y_{n_{k_L}} - u | \right\}
 \leq \| (y_{n_{k_L}-1} - u , y_{n_{k_L}} - u ) \|
 = \| T^{n_{k_L}}(\phi,\psi) - (u,u) \| 
$$
we have that for every $\epsilon > 0$ there exists 
$L \in \mathbb{N}$ such that
$|y_{n_{k_L}} - u |<\epsilon$ and 
$|y_{n_{k_L}-1} - u |<\epsilon$.
Since  $u$ is a locally asymptotically stable 
equilibrium for Eq.(\ref{eq: y2k2u})
by Lemma \ref{lemma: LAS}, it follows that 
$y_n \rightarrow u$.
This completes the proof of the lemma.
\end{proof}
\bigskip


\section{Proof of Theorem \ref{th: 3-2-L}}
\label{sec: Proof 3-2-L}
To prove Theorem \ref{th: 3-2-L} it is enough to assume 
statement (B) of 
Proposition \ref{prop: invariant and attracting}.
Also by Lemma \ref{lemma: p=q} we may assume $p \neq q$
without loss of generality.
Thus we make the following standing assumption, valid 
throughout the rest of this section for Eq.(3-2-L).
\bigskip

\noindent {\bf Standing Assumption (SA)}
{\it Assume $p\neq q$ and that there exist $m^*$, $M^*$ with $\mathcal{L} \leq m^* < M^* \leq \mathcal{U}$ such that 
for Eq.(3-2-L) and its associated function $f(x,y)$,
\begin{itemize}
\item[\rm (i)]  $[m^{*},M^{*}]$ is an invariant interval.
\item[\rm (ii)] every solution  eventually enters $[m^{*},M^{*}]$.
\item[\rm (iii)] $f(x,y)$ is coordinate-wise strictly monotonic on $[m^{*},M^{*}]^{2}$.
\end{itemize}
} 
The function $f(x,y)$ is assumed to be coordinate-wise monotonic on $[m^*,M^*]$,
and there are four possible cases in which this can happen:
(a) $f(x,y)$ is increasing in both variables,  
(b) $f(x,y)$ is decreasing in both variables,
(c) $f(x,y)$ is decreasing in $x$ and increasing in $y$, and
(d) $f(x,y)$ is increasing in $x$ and decreasing in $y$.

We present several lemmas before completing the proof of Theorem \ref{th: 3-2-L}.

\noindent 
By considering the restriction of the map $T$ of Eq.(3-2-L) to $[m^*,M^*]^2$,
an application of the Schauder  Fixed Point Theorem \cite{RBH} gives
that $[m^*,M^*]^2$ contains the fixed point of $T$, namely $(\overline{y},\overline{y})$.
Thus we have the following result. 
\begin{lemma}
\label{lemma: ybar in invariant interval}
$\overline{y} \in [m^*,M^*]$.
\end{lemma}

\begin{lemma}
\label{lemma: f(up,up) and f(down,down)}
Neither one of the systems of equations
$$
\mbox{\rm (S$_1$)}\quad 
\left\{
\begin{array}{rcl}
 M &=& f(M, M) \\ 
 m &=& f(m, m)
\end{array}
\right.
\quad \mbox{and} \quad 
\mbox{\rm (S$_2$)} \quad 
\left\{
\begin{array}{rcl}
 M &=& f(m, m) \\
  m &= & f(M, M)
\end{array}
\right.
$$
have solutions $(m,M) \in [m^*,M^*]^2$ with $m<M$.  
\end{lemma}
\begin{proof}
Since $x=\overline{y}$ is the only solution to $f(x,x)=x$,
it is clear that only $(\overline{y},\overline{y})$ satisfies (S$_1$).
Now let $(m,M)$ be a solution to (S$_2$).
From straightforward algebra applied to $M-m = f(m,m)-f(M,M)$ 
one arrives at $(p+1)(M-m)=0$, which implies $m=M$.
\end{proof}
\begin{lemma}
\label{lemma: 0<q<p}
Suppose $f(x,y)$ is increasing in $x$ and decreasing in $y$
for $(x,y) \in [m^*,M^*]$. Then $p-q>0$.
\end{lemma}
\begin{proof}
By the standing assumption (SA), $p \neq q$.
By Lemma \ref{lemma: signs of partials}, the coordinate-wise monotonicity hypothesis,
and the fact $\overline{y} \in [m^*,M^*]$ from Lemma \ref{lemma: ybar in invariant interval}, we have 
\begin{equation}
\label{eq: p-q sign}
(p-q)\,\overline{y} > q\,r \quad \mbox{\rm and} \quad (q-p)\,\overline{y} <  r
\end{equation}
The inequalities in (\ref{eq: p-q sign}) cannot hold simultaneously unless $p-q >0$.
\end{proof}
\begin{lemma}
\label{lemma: f up down lemma bounds}
If $f(x,y)$ is increasing in $x$ and decreasing in $y$
for $(x,y) \in [m^*,M^*]$, then 
$f([m^*,M^*]^2) \subset (1,\frac{p}{q})$.
\end{lemma}
\begin{proof}
For $(x,y) \in [m^*,M^*]$, the function $f$ is well defined and is
componentwise strictly monotonic on the set $[x,\infty) \times [y,\infty)$.  Then,
\begin{equation}
\label{eq: f up down infinity}
\begin{array}{l}
\displaystyle f(x,y) < \lim_{s\rightarrow \infty} f(s,y) =\lim_{s\rightarrow \infty} \frac{r+p\,s + y}{q\,s + y} = \frac{p}{q} \\ \\
\displaystyle f(x,y) > \lim_{t\rightarrow \infty} f(x,t) =\lim_{t\rightarrow \infty} \frac{r+p\,x + t}{q\,x + t} = 1
\end{array}
\end{equation}
\end{proof}
\begin{lemma}
\label{lemma: f up down iff}
Let $p>0$, $q>0$ and $r\geq 0$.
If $f(x,y)$ is increasing in $x$ and decreasing in $y$ on $[m^*,M^*]$,  
then
\begin{equation}
\label{eq: ybar in interval}
\frac{q\,r}{p-q} <   \frac{p}{q}
\end{equation}
\end{lemma}
\begin{proof}
Since $\overline{y} \in [m^*,M^*]$ by Lemma \ref{lemma: ybar in invariant interval}, we have
 $D_1(\overline{y},\overline{y}) > 0$ and $D_2(\overline{y},\overline{y})<0$. 
 By Lemma \ref{lemma: 0<q<p}, $p>q$, and 
by Lemma \ref{lemma: signs of partials},
\begin{equation}
\label{eq: signs 2}
(p-q)\,\overline{y} > q\,r
 \quad \mbox{\rm and} \quad 
 (q-p)\,\overline{y} <  r
\end{equation}
Then,
\begin{equation}
\label{eq: signs 3}
\overline{y} > \frac{q\,r}{p-q}
\end{equation}
In addition, by Lemma \ref{lemma: signs of partials}, 
\begin{equation}
\label{eq: f ybar up down infinity}
\displaystyle \overline{y} = 
f(\overline{y},\overline{y}) < \lim_{s\rightarrow \infty} f(s,\overline{y}) 
=\lim_{s\rightarrow \infty} \frac{r+p\,s + \overline{y}}{q\,s + \overline{y}} = \frac{p}{q} 
\end{equation}
\end{proof}
\begin{lemma}
\label{lemma: p-q+r>0}
Suppose $f(x,y)$ is increasing in $x$ and decreasing in $y$
for $(x,y) \in [m^*,M^*]$.
If $r<0$, then $p-q+r > 0$.
\end{lemma}
\begin{proof}
Since $D_1 f(x,y) >0$ for $(x,y) \in [m^*,M^*]$,
and by  
Lemma \ref{lemma: signs of partials}, 
Lemma \ref{lemma: ybar in invariant interval} and 
 by Lemma \ref{lemma: 0<q<p}, we have 
$\overline{y} > \frac{-r}{p-q}$, that is,
\begin{equation}
\label{eq: ybar bigger than}
\sqrt{(p+1)^2+4\,r\,(q+1)} > \frac{-2\,r\,(q+1)}{p-q}-p -1
\end{equation}
If the right-hand-side of inequality (\ref{eq: ybar bigger than}) is nonnegative,
then, after squaring both sides of (\ref{eq: ybar bigger than}) we have
\begin{equation}
\label{eq: ybar bigger than 2}
(p+1)^2+4\,r\,(q+1)  >  \left(\frac{-2\,r\,(q+1)}{p-q}\right)^2 +\frac{4\,r\,(q+1)(p+1)}{p-q}  + (p+1)^2
\end{equation}
Further simplification of (\ref{eq: ybar bigger than 2}) and the hypothesis $r<0$ yield
\begin{equation}
\label{eq: ybar biggern than 3}
1 < \frac{r(q+1)}{(p-q)^2} + \frac{p+1}{p-q}
\end{equation}
which, after some elementary algebra, implies $p-q+r > 0$.  
Now assume the right-hand-side of inequality (\ref{eq: ybar bigger than}) is negative, 
relation that we may rewrite as
\begin{equation}
\label{eq: RHS<0}
\frac{-r}{p-q} < \frac{1}{2} \, \frac{p+1}{q+1}
\end{equation}
If $\frac{1}{2} \, \frac{p+1}{q+1}\leq 1$, then $\frac{-r}{p-q} < 1$,which gives the conclusion $p-q+r>0$.
If $\frac{1}{2} \, \frac{p+1}{q+1} > 1$, that is, $p>2\, q+1$, then
\begin{equation}
\label{eq: p-q+r>q+r+1}
p-q+r > q+r+1
\end{equation}
Therefore if $q+r+1\geq 0$ the conclusion of the lemma follows
from this and from (\ref{eq: p-q+r>q+r+1}).
Assume 
\begin{equation}
\label{eq: assumption q+r+1<0}
q+r+1<0
\end{equation}
From relations (\ref{eq: pqr}) we have 
\begin{equation}
\label{eq: q+r+1<0}
q+r+1 = 
\frac{\left( b + c \right) \,
    \left( a^2\,c^2 + 
      b\,c^2\,\alpha  + 
      c^3\,\alpha  - 
      a\,c^2\,\beta  + 
      2\,a\,b\,c\,\gamma  + 
      a\,c^2\,\gamma  + 
      b^2\,{\gamma }^2 + 
      2\,b\,c\,{\gamma }^2 + 
      c^2\,{\gamma }^2 \right)
      }{c\,{\left( a\,c + 
        b\,\gamma  + 
        c\,\gamma  \right) }^2}
\end{equation}
hence assumption (\ref{eq: assumption q+r+1<0}) and relation (\ref{eq: q+r+1<0}) imply
\begin{equation}
\label{eq: long term}
R:= a^2\,c^2 + 
      b\,c^2\,\alpha  + 
      c^3\,\alpha  - 
      a\,c^2\,\beta  + 
      2\,a\,b\,c\,\gamma  + 
      a\,c^2\,\gamma  + 
      b^2\,{\gamma }^2 + 
      2\,b\,c\,{\gamma }^2 + 
      c^2\,{\gamma }^2 < 0
\end{equation}
Further algebra gives
\begin{equation}
\label{eq: further algebra}
\frac{\gamma}{a\,c}\, R - \left( -  c^2\,\alpha  
      + a\,c\,\gamma  - 
  c\,\beta \,\gamma  + 
  b\,{\gamma }^2 \right)  = 
c^2\,\alpha  + 
  \frac{b\,c\,\alpha \,
     \gamma }{a} + 
  \frac{c^2\,\alpha \,\gamma }
   {a} + b\,{\gamma }^2 + 
  c\,{\gamma }^2 + 
  \frac{2\,b\,{\gamma }^3}
   {a} + \frac{b^2\,
     {\gamma }^3}{a\,c} + 
  \frac{c\,{\gamma }^3}{a} > 0
\end{equation}
Since $R<0$ by (\ref{eq: long term}), from inequality (\ref{eq: further algebra}) we have 
\begin{equation}
\label{eq: one more}
-  c^2\,\alpha  
      + a\,c\,\gamma  - 
  c\,\beta \,\gamma  + 
  b\,{\gamma }^2
  < 0
  \end{equation}
Finally, from   (\ref{eq: pqr}) we have 
\begin{equation}
\label{eq: yet one more}
p-q+r = 
-  \frac{{\left( b + 
         c \right) }^2\,
      \left( -  c^2\,
           \alpha    + 
        a\,c\,\gamma  - 
        c\,\beta \,\gamma  + 
        b\,{\gamma }^2 \right)
        }{c\,
      {\left( a\,c + 
          b\,\gamma  + 
          c\,\gamma  \right) }
        ^2} 
\end{equation}
Combining (\ref{eq: one more}) with (\ref{eq: yet one more}) we obtain $p-q+r >0$.
\end{proof}
%

\begin{lemma}
\label{lemma: f up down ybar is GA}
If $r<0$ and $f(x,y)$ is increasing in $x$ and decreasing in $y$
for $(x,y) \in [m^*,M^*]$, then 
every solution converges to the equilibrium. 
\end{lemma}
\begin{proof}
Since $p-q+r>0$ by Lemma \ref{lemma: p-q+r>0}, we have $K_2 = -\frac{r}{p-q} < 1$, 
which together with  Lemma \ref{lemma: f up down lemma bounds}
implies that $[1,\frac{p}{q}]$ is an invariant, attracting compact interval
such that $f(x,y)$ is increasing in $x$ and decreasing in $y$ on $[1,\frac{p}{q}]^2$.
Since $f([1,\frac{p}{q}]^2) \subset (1,\frac{p}{q})$, we see that every solution to Eq.(3-2-L) eventually enters 
the invariant interval $(1,\frac{p}{q})$.
The change of variables
\begin{equation}
\label{eq: change of vars KocicLadas}
y_n = \frac{1+\frac{p}{q} \, z_n}{1+z_n}\, , 
\quad \mbox{or} \quad 
z_n = \frac{x_n-1}{\frac{p}{q}-x_n}\, ,
\end{equation}
  transforms the equation
\begin{equation}
 \label{eq: old KL}
 y_{n+1} = \frac{r+p \, y_n + y_{n+1}}{q\, y_n+y_{n+1}}, 
 \quad n=0,1,\ldots, \quad y_{-1},y_0 \in \left(1,\frac{p}{q}\right)
\end{equation}
 into the equivalent equation
 \begin{equation}
 \label{eq: new KL}
 z_{n+1} = g(z_n,z_{n-1}), 
 \quad n=0,1,\ldots, \quad z_{-1},z_0 \in (0,\infty)
 \end{equation}
 where
 $$
 g(w,v) := \frac{q(1+v)(-\left( q\,\left( p - q + r
       \right)  \right)  + \left( -p^2 + p\,q - q\,r\right) \,w)}
       {(1+w)(-  q\,\left( p - q - q\,r \right)    + 
  \left( -p^2 + p\,q + q^2\,r \right) \,v)}
 $$
We claim that for $w, v \in (0,\infty)$,  
(a) $g(w,w)$ is increasing in $w$,
(b) $g(w,v)/w$ is decreasing in $w$, and 
(c) $g(w,v)/w$ is decreasing in $v$.  Indeed, since $p>q$, $r<0$, $p-q+r>0$, and 
$\frac{-r}{p-q} < \frac{p}{q}$ we have
$$
\frac{d}{dw} \left(g(w,w) \right) =   \frac{-r\, q^2\, \left( 1 + q \right) \, {\left( -p + q \right) }^2 }
        {{\left( - p\,q    + q^2 + q^2\,r - p^2\,w + p\,q\,w + q^2\,r\,w \right) }^2} 
         > 0
$$
$$
\frac{\partial}{\partial w}
\left(\frac{g(w,v)}{v}\right) = 
- \frac{q\,{\left( -p + q \right)}^2\,\left(    q\,
\left( p - q + r \right)     + \left(  p^2 - p\,q + q\,r \right) \,w\right) }
      {{\left( - p\,q   + q^2 + q^2\,r - p^2\,v + p\,q\,v + q^2\,r\,v\right) }^2\,w\,\left( 1 + w \right) }
    < 0 
$$
$$
\frac{\partial}{\partial w}
\left(\frac{g(w,v)}{w}\right) = 
-  \frac{q\,
      \left( 1 + v \right) \,
      \left(    q\,
           \left( p - q + r
           \right)     +
        2\,q\,
         \left( p - q + r
           \right) \,w + 
        \left( p^2 - p\,q + 
           q\,r \right) \,w^2
        \right) }{\left(   
           q\,
           \left( p - q - 
          q\,r \right)    
         + \left(  p^2 - p\,q - 
           q^2\,r \right) \,v
        \right) \,w^2\,
      {\left( 1 + w \right) }^2}
    <  0
$$
Also, note that Eq.(\ref{eq: new KL}) has a unique equilibrium $\overline{z}$.
 Therefore hypotheses (1)--(4) of Theorem \ref{th: Kocic-Ladas}  are satisfied, so
 every solution $\{z_n\}$ to Eq.(\ref{eq: new KL}) converges to 
 $\overline{z}$.  By reversing the change of variables, one can conclude that
 every solution to Eq.(\ref{eq: old KL}) converges to the equilibrium.
 \end{proof}
 
 \bigskip

\noindent {\bf Proof of Theorem \ref{th: 3-2-L}.}  
The four parts of the proof are:
\begin{itemize}
\item[\rm a.] 
{\it $f(x,y)$ is increasing in both $x$ and $y$ on $[m^*,M^*]^2$:}
By Lemma \ref{lemma: f(up,up) and f(down,down)}
the hypotheses of Theorem \ref{th: mM first} part (i).
is satisfied, hence every solution converges to the 
equilibrium $\overline{y}$.
\item[\rm b.] 
{\it $f(x,y)$ is decreasing in both $x$ and $y$ on $[m^*,M^*]^2$:}
By Lemma \ref{lemma: f(up,up) and f(down,down)}
the hypotheses of Theorem \ref{th: mM first} part (ii).
is satisfied, hence every solution converges to the 
equilibrium $\overline{y}$.
\item[\rm c.] 
{\it $f(x,y)$ is decreasing in $x$ and increasing in $y$ on $[m^*,M^*]^2$:}
By the corollary to Theorem \ref{th: Camouzis-Ladas}  
we conclude every solution converges to the unique equilibrium 
  or to a prime period-two solution.
\item[\rm d.] 
{\it $f(x,y)$ is increasing in $x$ and decreasing in $y$ on $[m^*,M^*]^2$:}
By Lemmas \ref{lemma: signs of partials},  \ref{lemma: 0<q<p}, and \ref{lemma: f up down lemma bounds}, there is no loss of generality
in assuming $[m^*,M^*] \subset (K,\frac{p}{q})$, where
$K := \max\{ \frac{-r}{p-q},\frac{q\, r}{p-q}\}$, which we do.
We consider two subcases.
If $r \geq 0$, then Lemma \ref{lemma: 0<q<p}, 
Lemma \ref{lemma: f up down iff} 
and Proposition \ref{prop: f up down r >0 on I}
imply that every solution converges to the unique equilibrium.
If $r<0$, then Lemma \ref{lemma: f up down ybar is GA} 
implies that every solution converges to the unique equilibrium.
\end{itemize}
This completes the proof of Theorem \ref{th: 3-2-L}.
Since Theorem \ref{th: 3-2-L} is just a version of Theorem \ref{th: 3-3}
obtained by an affine change of coordinates,
we have also proved Theorem \ref{th: 3-3} as well.
\hfill $\Box$

\section{Proof of Theorem \ref{th: 3-2}}
\label{sec: Proof 3-2}
The first lemma guarantees solutions to Eq.(3-2) to be bounded.
\begin{lemma}
\label{lemma: 3-2 is bounded}
Let $p>0$, $q>0$ and $r \geq 0$.
There exist positive constants $\mathcal{L}$ and $\mathcal{U}$ such that 
every solution $\{x_n\}_{n=-1}^\infty$ to Eq.(3-2) satisfies 
$x_n \in [\mathcal{L},\mathcal{U}]$ for $n\geq 2$, and 
the function 
$$
f(x,y) = \frac{r+p\,x+y}{q\,x+y}, \quad (x,y) \in (0,\infty)^2
$$
satisfies 
$$
f([\mathcal{L},\mathcal{U}]\times [\mathcal{L},\mathcal{U}]) \subset [\mathcal{L},\mathcal{U}]
$$
\end{lemma}
\begin{proof}
Set 
$$
\mathcal{L} := \min \left\{\frac{p}{q},1\right\}, \quad 
\mbox{\rm and} \quad 
\mathcal{U} := \max \left\{ \frac{p}{q},1, \frac{r+(p+1) \mathcal{L} }{(q+1)\,\mathcal{L}}\right\}
$$
Since 
$$
0 \leq r + (p - q \mathcal{L})\,x + (1-\mathcal{L})\, y   \quad 
\mbox{for} \ (x,y) \in (\mathcal{L},\infty)^2\, ,
$$
then  
$\mathcal{L} \, q \, x + \mathcal{L}\, y \leq r + p\,x+q\,y$ for $x,y  \geq \mathcal{L}$,, i.e., 
$$
\mathcal{L} \leq \frac{r+p\,x+y}{q\,x+y}, \quad (x,y) \in (0,\infty)^2
$$
From the definition of $\mathcal{U}$ we have
$$
-r + (q\, \mathcal{U} - p) \mathcal{L} + ( \mathcal{U}-1) \mathcal{L} \geq 0
$$
Write $x,\, y \in [\mathcal{L},\infty)$ as 
$x = \mathcal{L}+v$, $y=\mathcal{L}+w$ for $v,\, w\in [0,\infty)$.  
Then for $v,w \in [0,\infty)$,
$$
\begin{array}{l}
-r + (q\, \mathcal{U}-p) \, x + (\mathcal{U}-1)\, y \\ \\
\quad = -r + (q\, \mathcal{U}-p) \, (\mathcal{L}+v) + (\mathcal{U}-1)\, (\mathcal{L}+w) \\ \\
\quad = -r + (q \, \mathcal{U} - p ) \, \mathcal{L} + (\mathcal{U}-1) \, \mathcal{L} + 
(q \, \mathcal{U}-p) \, v + (\mathcal{U}-1)\, w \\ \\
\quad  \geq 0 
\end{array}
$$
that is,
$$
\frac{r+p\,x+y}{q\,x+y} \leq \mathcal{U} \quad \mbox{for } x,y \in [\mathcal{L},\infty)
$$
\end{proof}

Inspection of the proof of Proposition \ref{prop: invariant and attracting} reveals 
that, given that we have Lemma \ref{lemma: 3-2 is bounded},
the conclusion of the proposition is true concerning Eq.(3-2).
The statement is given next.
\begin{proposition}
\label{prop: invariant and attracting 3-2}
At least one of the following statements is true:
\begin{itemize}
\item[\rm (A)]  
Every solution to Eq.(3-2) converges to the  equilibrium.
\item[\rm (B)]   
There exist  $m^*$, $M^*$ with \  $L\, \leq m^*< M^*$ s.t. 
\begin{itemize}
\item[\rm (i)]  $[m^{*},M^{*}]$ is an invariant interval for Eq.(3-2), i.e., 
$f([m^{*},M^{*}]\times [m^{*},M^{*}]) \subset [m^{*},M^{*}]$.
\item[\rm (ii)] Every solution  to Eq.(3-2) eventually enters $[m^{*},M^{*}]$.
\item[\rm (iii)] $f(x,y)$ is coordinate-wise strictly monotonic on $[m^{*},M^{*}]^{2}$.
\end{itemize}
\end{itemize}
\end{proposition}
The proof of Theorem \ref{th: 3-2-L} may be reproduced here in its entirety
with the only change being
the elimination of the case $r<0$, which presently does not apply.
Everything else in the proof applies to Eq.(3-2).
The proof of Theorem \ref{th: 3-2} is complete.
\newpage

\section{Appendix: Computer Algebra System Code}
\label{appendix: computer 1}
\begin{table}[!ht]
\centerline{\includegraphics[width=5in]{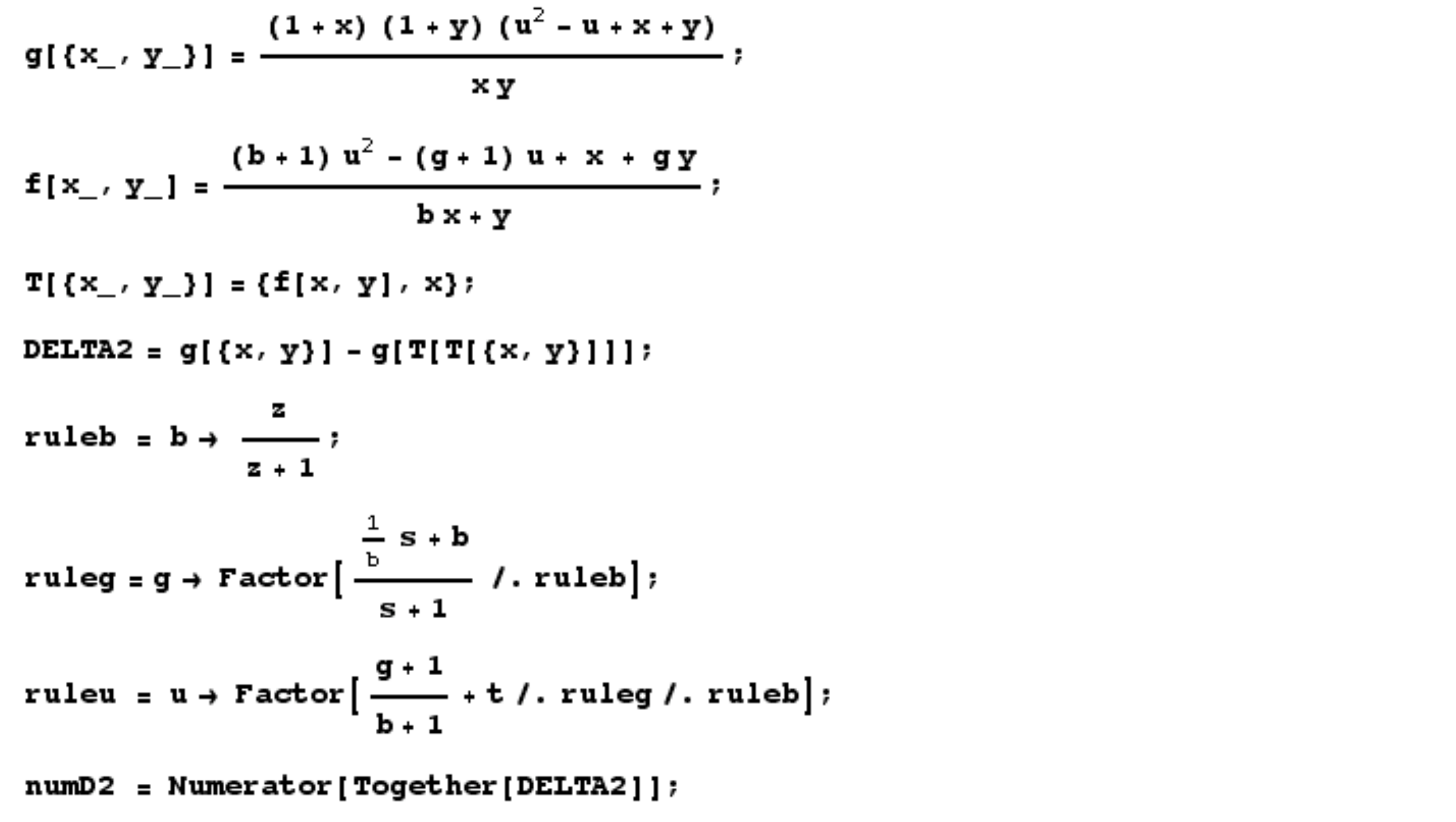}}
\centerline{\includegraphics[width=5in]{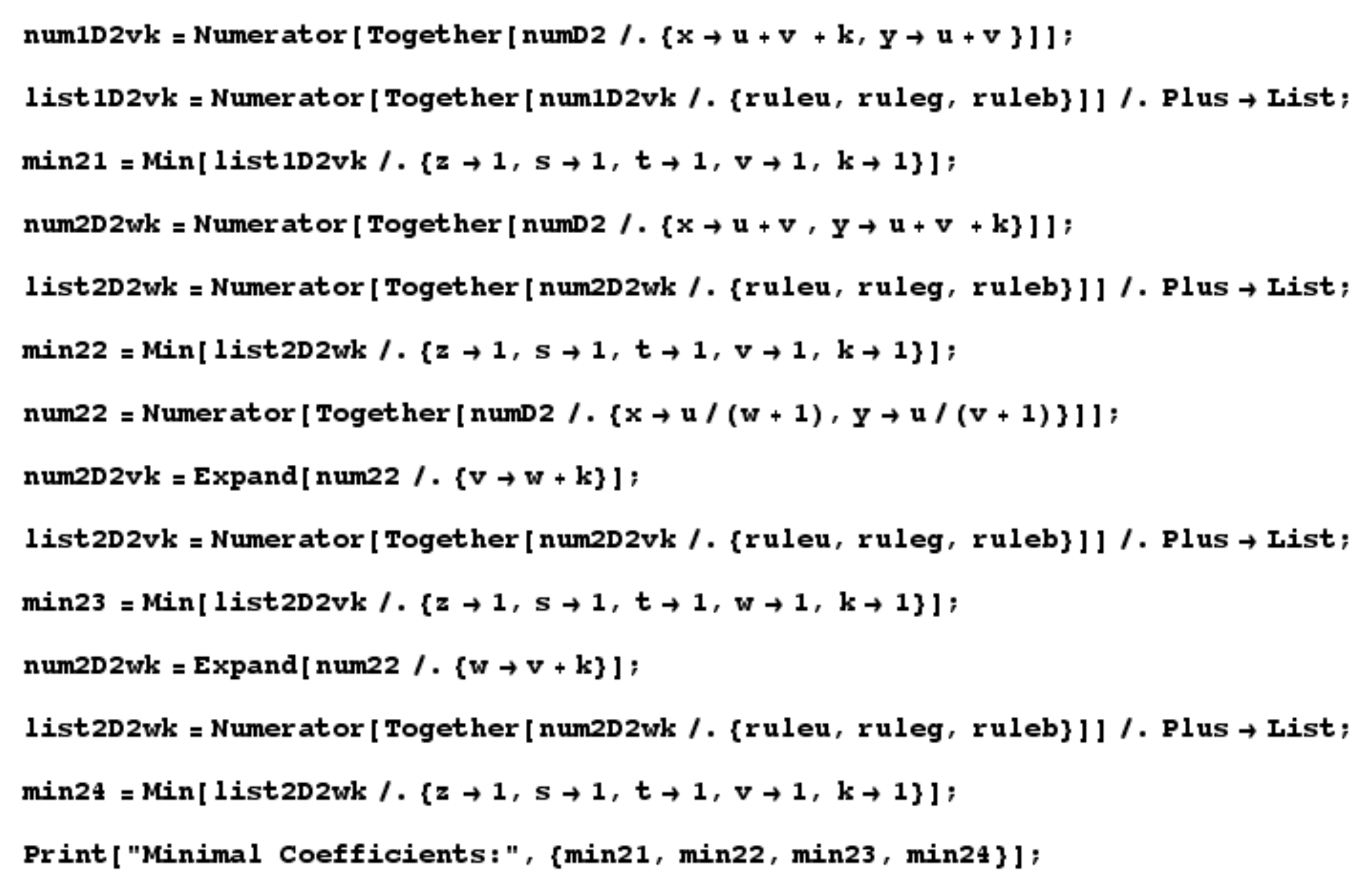}}
\caption{Mathematica code needed to do the calculations in Claim \ref{claim: Q1Q3, g>b}.
Here we define the functions $g$, $f$ and $T$, as well as the expression DELTA2.
The reparametrizations indicated in the proof of Claim \ref{claim: Q1Q3, g>b} for the case
$g > b$ are defined as  substitution rules. 
To verify the positive sign of a polynomial of nonnegative variables $z$, $s$, \ldots,
we form a list with the terms of the polynomial, and then substitute the number $1$ 
for the variables in order to extract the smallest coefficient.
  This input was tested on Mathematica Version 5.0
\cite{Mathematica}.}
\label{table: 1}
\end{table}
\begin{table}[!ht]
\centerline{\includegraphics[width=5in]{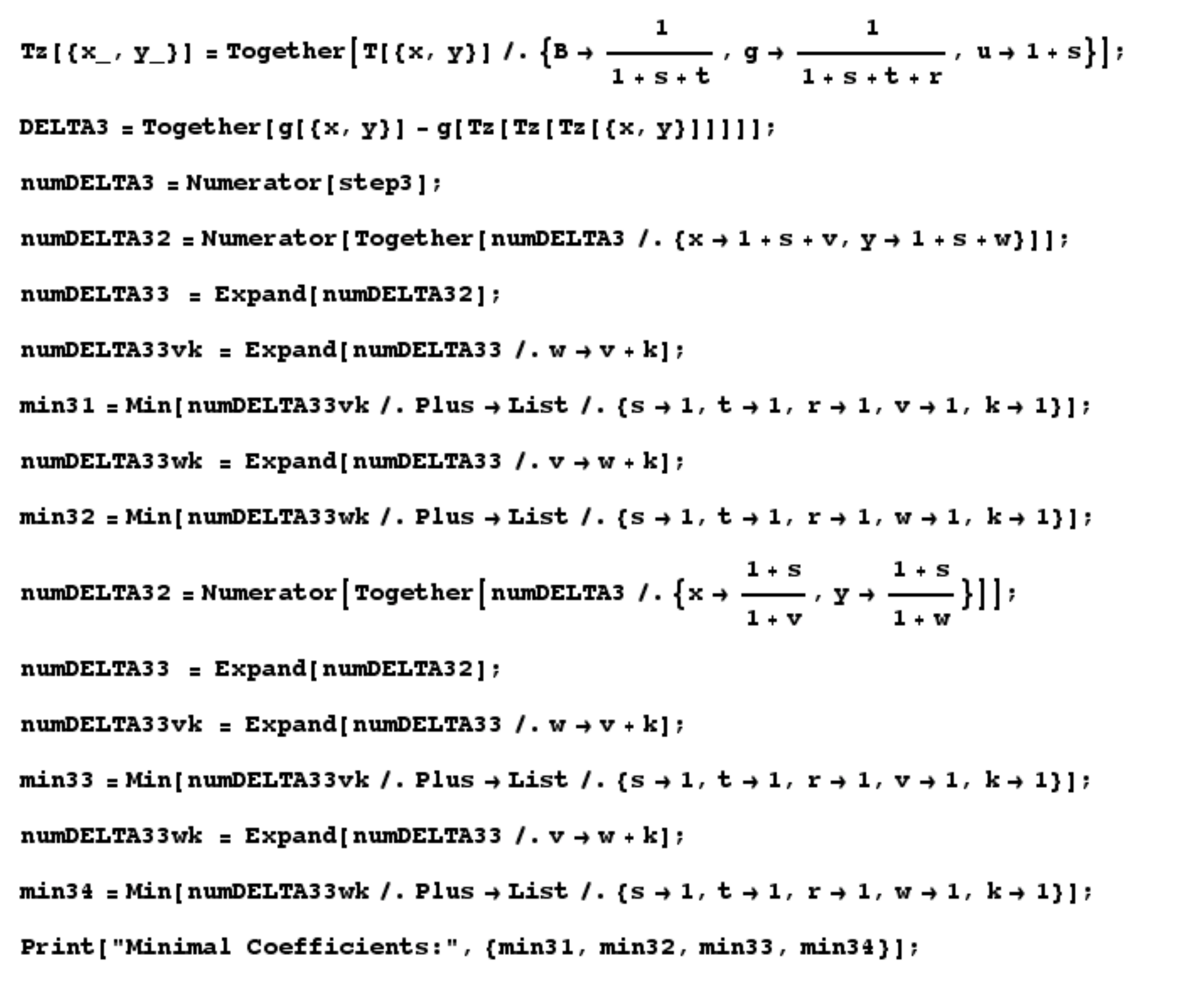}}
\caption{
Mathematica code needed to do the calculations in Claim \ref{claim: Q1Q3, g>b}
when $g\leq b$  The functions $g$, $f$ and $T$ are defined as before (not shown).  
  This input was tested on Mathematica Version 5.0
\cite{Mathematica}.}
\label{table: 3}
\end{table}

\newpage

\end{document}